\documentclass[letterpaper, 10 pt, conference]{ieeeconf}

\IEEEoverridecommandlockouts                              
\title{Convergence Rate of a Penalty Method for Strongly Convex Problems
with Linear Constraints}

\author{Angelia Nedi\'c and Tatiana~Tatarenko
\thanks{A.\ Nedi\'c (\url{Angelia.Nedich@asu.edu}) is with School of Electrical, Computer and Energy Engineering,
        Arizona State University, USA, and T.\ Tatarenko (\url{tatiana.tatarenko@rmr.tu-darmstadt.de}) is with the Control Methods and Robotics Lab., Technical University Darmstadt, Darmstadt, Germany.}
}

\usepackage{euscript}
\usepackage{amsmath, amssymb, epsfig, graphicx, bm}
\usepackage{cases}
\usepackage{float, subfigure}
\usepackage{color}
\usepackage{amsthm}
\usepackage{amsfonts}
\usepackage{psfrag}
\usepackage{cite, url}
\usepackage{breqn}
\usepackage{overpic}
\usepackage[all,cmtip]{xy}
\usepackage{color,hyperref}
\definecolor{darkblue}{rgb}{0.0,0.0,0.6}
\hypersetup{colorlinks,breaklinks,
	linkcolor=darkblue,urlcolor=darkblue,
	anchorcolor=darkblue,citecolor=darkblue}

 \newtheorem{cor}{Corollary}
 \newtheorem{prop}{Proposition}

 \newtheorem{lem}{Lemma}

\def\0{{\bf 0}}
\def\1{{\bf 1}}

\def\bes{\begin{equation*}}
	\def\ees{\end{equation*}}
\def\be{\begin{equation}}
	\def\ee{\end{equation}}
\def\bas{\begin{eqnarray*}}
	\def\eas{\end{eqnarray*}}
\def\ba{\begin{eqnarray}}
	\def\ea{\end{eqnarray}}
\def\bma{\begin{bmatrix}}
	\def\ema{\end{bmatrix}}
\def\bmx{\begin{matrix}}
	\def\emx{\end{matrix}}
\def\ben{\begin{enumerate}}
	\def\een{\end{enumerate}}
\def\bit{\begin{itemize}}
	\def\eit{\end{itemize}}
\def\bet{\begin{tabular}}
	\def\eet{\end{tabular}}

\def\de{\delta}

\def\an#1{{\color{black}#1}}

\def\R{\mathbb{R}}
\def\E{\mathbb{E}}
\def\d{\delta}
\def\la{\langle}
\def\ra{\rangle}
\def\b{{\beta}}
\def\a{\alpha}
\def\g{\gamma}

\def\dist{{\rm dist}}

\def\ex#1{\mathop \mathbb{E}\left[ {#1}\right] }

\begin{document}

\maketitle

\begin{abstract}
We consider an optimization problem with strongly convex objective and linear inequalities constraints. To be able to deal with a large number of constraints we provide a
penalty reformulation of the problem. As penalty functions we use  a version of the one-sided Huber losses.
The smoothness properties of these functions allow us to choose time-varying penalty parameters in such a way
that the incremental procedure with the diminishing step-size converges to the exact solution with the rate $O(1/{\sqrt k})$. To the best of our knowledge, we present the first result on the convergence rate for the penalty-based gradient method, in which the penalty parameters vary with time.

\end{abstract}


\section{Introduction}
In  this  paper,  we  study  the  problem  of minimizing  a
{\it convex} function $f:\R^n\to\R$ over a convex and closed set $X$
that is the intersection of finitely many {\it convex and closed sets} $X_i$, $i=1,\ldots,m$
(where $m\ge 2$ is large),
i.e.,
\ba\label{eq:gprob}
&&\hbox{minimize \ \,} \quad  f(x) \cr
&&\hbox{subject to } \quad  x\in X = \cap_{i=1}^m X_i.
\ea
Throughout the paper, the function $f$ is assumed to be {\it $\mu$-strongly convex} over $\R^n$.
Optimization problems of the form~\eqref{eq:gprob} arise in many areas of research,
such as digital filter settings in communication systems \cite{filter}, energy consumption in Smart Grids \cite{SmartG}, convex  relaxations of  various  combinatorial  optimization  problems in machine learning applications \cite{clustering, matching}.

Our interest is in case when $m$ is large, which prohibits us from using
projected gradient and
augmented Lagrangian methods~\cite{BertsekasConstrOpt},
which require either computation of the (Euclidean) projection or
an estimation of the gradient for the sum of many functions, at each iteration.
To reduce the complexity, one may consider a method that operates on a  single set $X_i$
from the constraint set collection $\{X_1,\ldots,X_m\}$ at each iteration.
Algorithms using random constraint sampling for general convex optimization problems~\eqref{eq:gprob}  have been first considered  in~\cite{Nedich2011} and were extended in~\cite{WangBertsSM} to a broader class of randomization over the sets of constraints. Moreover, the convergence rate analysis is performed
in~\cite{WangBertsSM} to demonstrate that the optimality error diminishes to zero with the rate of $O(1/\sqrt k)$.

In this work, we present an alternative penalty-based approach to guarantee convergence to the optimum while processing a single set $X_i$ per iteration.
A possible reformulation of the problem~\eqref{eq:gprob} is through the use of the indicator functions of
the constraint sets, resulting in the following unconstrained problem
\be\label{eq:reform}
\min_{x\in\R^n}\sum_{i=1}^m \left\{\frac{1}{m}f(x) + \chi_i(x)\right\},
\ee
where $\chi_i(\cdot):\R^n\to\R\cup\{+\infty\}$ is the indicator function of the set $X_i$
(taking value $0$ at the points
$x\in X_i$ and, otherwise, taking value $+\infty$).
The advantage of this reformulation is that the objective function is the sum of convex functions and incremental methods can be employed that compute only a (sub)-gradient of one of the component functions at each iteration.
The traditional incremental methods do not have memory, and their origin can be traced back to work of Kibardin~\cite{Kibardin}. They have been studied for
smooth least-square problems~\cite{Ber96,Ber97,Luo91}, for training the neural networks~\cite{Gri94,Gri00,LuT94}, for smooth convex problems~\cite{Sol98,Tse98} and
for non-smooth convex problems~\cite{GGM06,HeD09,JRJ09,Kiw04,NeB00,NeB01,NeBBor01,Wright08}
(see~\cite{BertsekasPenalty} for a more comprehensive survey of these methods).
However, no rate of convergence to the exact solution has been obtained for such procedures.
Reformulation~\eqref{eq:reform} has been considered in~\cite{Kundu2017} as a departure point
toward an exact penalty reformulation using the set-distance functions.
This exact penalty formulation has been motivated by a simple exact penalty model proposed in~\cite{Bertsekas2011} (using only the set-distance functions)
and a more general penalty model considered in~\cite{BertsekasPenalty}.
In~\cite{Kundu2017}, a lower bound on the penalty parameter
has been identified guaranteeing that the optimal solutions of the penalized problem are also optimal solutions of the original problem~\eqref{eq:reform}. However, this bound depends on a so-called  regularity constant for the constraint set, which might be difficult to estimate. Moreover, the proposed approaches in~\cite{Kundu2017} do not utilize
incremental processing, but rather primal-dual approaches where a full (sub)-gradient of the penalized function
is used.

In contrast to the works mentioned above, this paper deals with a penalized reformulation of the problem~\eqref{eq:gprob}, where the penalty parameter can be gradually increased to guarantee convergence of the incremental procedure to the exact solution. The corresponding penalty functions correspond to a version of the one-sided Huber losses~\cite{Huber}, which are smooth and possess Lipschitz continuous gradients.
In our previous work~\cite{Penalty_arxiv}, we have demonstrated existence of the settings for this penalized reformulation under which the fast incremental algorithms can be applied to achieve convergence to a predefined feasible neighborhood of the optimum with a linear rate. However, to guarantee this convergence, we need to know some problem specific parameters. These parameters might be difficult to estimate in practice.
That is why in this work, we study some new properties of these penalty functions which allow us to set up the time-dependent parameters of the reformulated unconstrained problem such that convergence to the exact optimum with the average rate $O(1/{\sqrt k})$ is guaranteed. To the best of our knowledge, this is the first result on the convergence rate for the penalty-based optimization with time-varying parameters.

\section{Problem formulation and its penalty-based reformulation}
We consider the following optimization problem:
\ba\label{eq:problem}
&&\hbox{minimize \ \,} \quad  f(x) \cr
&&\hbox{subject to } \quad  \la a_i,x\ra - b_i\le 0, \  i=1,\ldots,m,
\ea
where the vectors $a_i$, $i=1,\ldots, m$, are nonzero.
We will assume that the problem is {\it feasible}.
Associated with problem~\eqref{eq:problem}, we
consider a penalized problem
\ba\label{eq:pen-problem}
&&\hbox{minimize \ \,} \quad  F_{\g\d}(x) \cr
&&\hbox{subject to }  \quad x\in\R^n,
\ea
where
\begin{equation}\label{eq:penfun}
F_{\g\d}(x) = f(x) + \frac{\gamma}{m} \sum_{i=1}^{m} h_\d\left(x; a_i, b_i\right).\end{equation}
Here, $\gamma>0$ and $\d\ge0$ are penalty parameters.
The vectors $a_i$ and scalars $b_i$ are the same as those characterizing the constraints in problem~\eqref{eq:problem}.
For a given nonzero vector $a\in\R^n$ and $b\in\R$, the penalty function
$h_\d(\cdot;a,b)$ is given by\footnote{A version of the one-sided Huber losses~\cite{Huber}.}
\begin{align}\label{eq:hfun}
h_\delta(x; a,b) = \begin{cases}
\frac{\la a,x\ra -b}{\|a\|} &\text{ if } \  \la a,x\ra - b>\de,\\
\frac{(\la a,x\ra -b + \de)^2}{4\de\|a\|} &\text{ if } \  -\de\le\la a,x\ra - b\le\de,\\
 0 &\text{ if } \  \la a,x\ra - b<-\de,
\end{cases}
\end{align}
\an{(see Figure~\ref{fig:penalty} for an illustration).}
For any $\d\ge0$, the function $h_\d(x;a,b)$ satisfies the following relations:
\begin{align}\label{eq:hfunineq}
h_\delta(x; a,b) \ge 0\qquad\hbox{for all }x\in\R^n,
\end{align}
\begin{align}\label{eq:hfunineq1}
h_\delta(x; a,b) \le \frac{\delta}{4\|a\|},\qquad\hbox{when }\la a,x\ra\le b,
\end{align}
\begin{align}\label{eq:hfunineq2}
h_\delta(x; a,b) > \frac{\delta}{4\|a\|},\qquad\hbox{when }\la a,x\ra> b.
\end{align}
\begin{figure}[!t]
\centering
\psfrag{-0.5}[c][l]{\scriptsize{$-0.5$}}
\psfrag{0}[c][l]{\scriptsize{$0$}}
\psfrag{0.5}[c][l]{\scriptsize{$0.5$}}
\psfrag{1}[c][l]{\scriptsize{$1$}}
\psfrag{1.5}[c][l]{\scriptsize{$1.5$}}
\psfrag{2}[c][l]{\scriptsize{$2$}}
\psfrag{h}[c][l]{{\Large{$h_{\delta}$}}}
\psfrag{x}[c][b]{\Large{$x$}}
\begin{overpic}[width=1\linewidth]{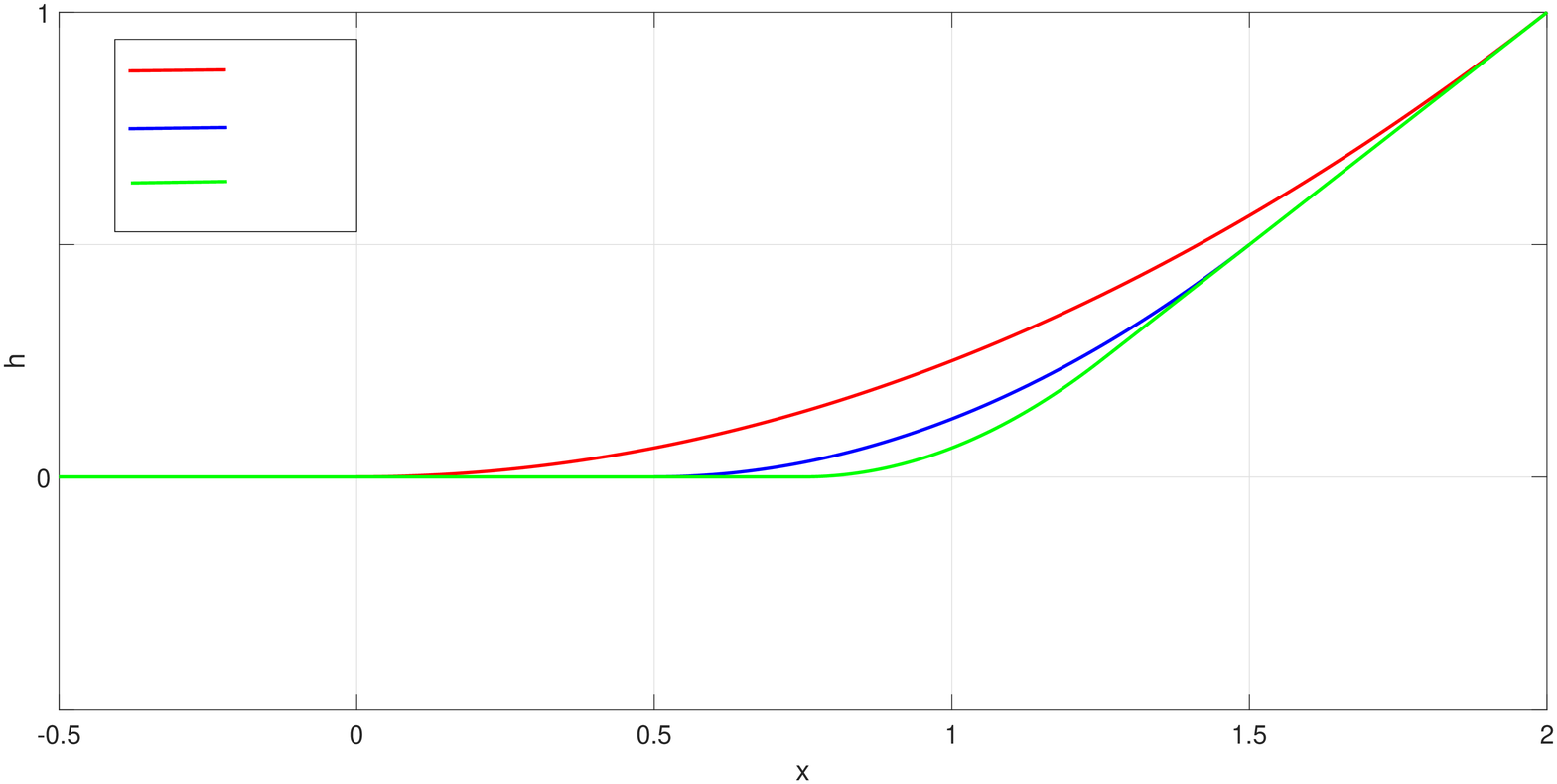}
\put(23,37.5){\scriptsize{$1$}}
\put(23,34){\scriptsize{$\frac12$}}
\put(23,30.5){\scriptsize{$\frac14$}}
\end{overpic}
\caption{Penalty functions $h_\delta(x;1,1)$ for the constraint $x-1\le 0$, $x\in\R$,
with $\d\in\left\{\frac{1}{4},\frac{1}{2},1\right\}$.}
\label{fig:penalty}
\end{figure}

Observe that $h_\delta(x; a,b)$ can be viewed as a composition of
a scalar function
\begin{align}\label{eq:sfun}
p_\d(s)= \begin{cases}
s &\text{ if } \quad s>\d,\\
\frac{(s + \de)^2}{4\d} &\text{ if } \quad -\de\le s \le\de,\\
 0 &\text{ if } \quad s<-\de,
\end{cases}
\end{align}
with a linear function  $x\mapsto \la a,x\ra-b$, which is scaled by $\frac{1}{\|a\|}$.
In particular, we have
\be\label{eq:handp}
h_\delta(x; a,b) =\frac{1}{\|a\|}p_\d(\la a,x\ra-b).\ee
The function $p_\d(s)$ is convex on $\R$ for any $\delta\ge0$.
Thus, the function $h_\d(x;a,b)$ is convex on $\R^n$, implying that
the objective function~\eqref{eq:penfun} of the penalized problem~\eqref{eq:pen-problem}
is convex over $\R^n$ for any $\delta\ge0$ and $\g>0$.

Furthermore, the function
$p_\d(\cdot)$ is twice differentiable for any $\d>0$,
with the first and second derivatives given by
\begin{align}\label{eq:pderiv}
p'_\d(s)= \begin{cases}
1 &\text{ if } \quad s>\d,\\
\frac{(s + \de)}{2\d} &\text{ if } \quad -\de\le s \le\de,\\
 0 &\text{ if } \quad s<-\de,
\end{cases}
\end{align}
\begin{align*}
p''_\d(s) = \begin{cases}
\frac{1}{2\d} &\text{ if } \quad -\d\le s\le\d,\\
 0 &\text{ if } \quad s<-\de \quad \text{or}\quad s>\d.
\end{cases}
\end{align*}
Thus, the function $p(s)$ has Lipschitz continuous derivatives with constant $\frac{1}{2\d}$.
Hence, the function $h_\delta(\cdot; a,b)$ is differentiable for any $\d>0$, and
its gradient is given by
\be\label{eq:gradh}
\nabla h_\delta(x; a,b) =\frac{1}{\|a\|}\,p'_\d(\la a,x\ra-b) a,\ee
which is Lipschitz continuous with a constant $\frac{\|a\|}{2\d}$, i.e.,
\be\label{eq:Lipc-gradh}
\|\nabla h_\delta(x; a,b) -\nabla h_\delta(y; a,b)\| \le \frac{\|a\|}{2\d}\,\|x-y\|
\ee
for all $x,y\in\R^n$.
In view of the definition of the penalty function $F_{\g\d}$ in~\eqref{eq:penfun}
and relation~\eqref{eq:gradh},
we can see that the magnitude of the ``slope" of the penalty function is controlled by the parameter $\g>0$,
while the ratio of the parameters $\g$ and $\d$ is controlling the ``curvature" of the penalty function.

Our choice of the penalty function is motivated by a desire to have the minimizers of
the penalized problem~\eqref{eq:pen-problem} being feasible
for the original problem~\eqref{eq:problem}.
Note that the penalty function proposed above is a version of the one-sided Huber losses.
Originally, the Huber loss functions were introduced in applications of robust regression models
to make them less sensitive to outliers in data in comparison with the squared error loss~\cite{Huber}.
In contrast,
we use this type of penalty function to smoothen the exact penalties based on the distance to the sets $X_i$
proposed in~\cite{BertsekasPenalty}.
Furthermore, an appropriate choice of the parameter $\delta\ge0$ allows us to overcome the limitation
of the smooth penalties based on the squared distances to the sets $X_i$, which typically
provide an infeasible solution (for the original problem), due to a small penalized value around an optimum lying close to the feasibility set boundary~\cite{Siedlecki}.

In what follows, we let $\Pi_Y [x]$ denote the (Euclidean) projection of a point $x$ on a convex closed set $Y$,
i.e.,
\[\dist(x, Y) = \|x - \Pi_{Y} [x]\|.\]

The following lemma and its corollary provide some additional
properties of the penalty function $h_\d(x;a,b)$ that we will use later on.
The proof can be found in~\cite{Penalty_arxiv}.

\begin{lem}\label{lem:penalty}
Given a nonzero vector $a\in\R^n$ and a scalar $b\in\R$, consider the penalty function $h_\d(x;a,b)$ defined in~\eqref{eq:hfun} with $\d\ge0$.
Let $Y=\{x\mid \la a,x\ra-b\le0\}.$
Then, we have for $\d=0$,
\[h_0(x;a,b)=\dist(x,Y) \qquad\hbox{for all $x\in\R^n$},\]
and for any $0<\d \le \d'$,
\[h_\d(x;a,b) \le h_{\d'}(x;a,b) \quad\hbox{for all $x\in\R^n$}.\]
\end{lem}

The following corollary shows that choosing $f(\hat x)$, for any
feasible $\hat x$, can be used to construct non-empty level sets of $F_{\g\d}$ and $f$.
\begin{cor}\label{cor:lset}
Let $\g>0$ and $\d\ge 0$ be arbitrary, and let
$\hat x$ be a feasible point for the original problem~\eqref{eq:problem}.
Then, for the scalar $t_{\g\d}(\hat x)$ defined by
\[t_{\g\d}(\hat x)= f(\hat x)+\frac{\g\d}{4\min_{1\le i\le m}\|a_i\|},\]
the level set
$\{x\in\R^n\mid F_{\g\d}(x)\le t_{\g\d}(\hat x)\}$ is nonempty and
the solution set $X^*_{\g\d}$ of
the penalized problem~\eqref{eq:pen-problem} is contained in the level set
$\{x\in\R^n\mid f(x)\le t_{\g\d}(\hat x)\}$.
\end{cor}

In the next section, we will consider the settings for the penalty parameters under which the incremental gradient-based procedure for the unconstrained problem~\eqref{eq:pen-problem} leads to the solution of the original constrained problem~\eqref{eq:problem}.
Moreover, we will establish the convergence rate of this procedure.

\section{Penalized optimization with time-varying parameters}
We consider sequences $\{\delta_k\}$ and $\{\g_k\}$ of positive scalars,
and we denote the corresponding penalty
function $F_{\d_k\g_k}(x)$ simply by $F_k$,
i.e.,
\begin{equation}
F_k(x) = f(x) + \frac{\g_k}{m} \sum_{i=1}^{m} h_k\left(x; a_i, b_i\right),
\end{equation}
where we use $h_k$ to denote the function $h_{\d_k}$.
When $f$ is strongly convex, each of these penalty functions has a unique solution,
denoted by $x_k^*$, and the original problem also has a unique solution $x^*\in X$.

First, we derive an upper bound for the distance between $x_k^*$ and  $x_{k+1}^*$.
To provide such a bound, we use some properties of the gradients of $h_k$,
as given in the following lemma.

\begin{lem}\label{lem:pderiv}
Consider the function $h_\d(\cdot; a,b)$ as given in~\eqref{eq:hfun}.
Then, we have
\[\|\nabla h_\d(x;a,b)\|\le 1\qquad\hbox{for all }x\in\R^n.\]
If $\d_1 \ge \d_2$, then
\[\max_{x\in\R^n}\|\nabla h_{\d_1}(x;a,b)- \nabla h_{\d_2}(x;a,b)\|
\le \frac{\d_1-\d_2}{2\d_1}.\]
\end{lem}
\begin{proof}
In view of relation~\eqref{eq:gradh} we have that
\[\|\nabla h_\d(x;a,b)\|=|p'_\d(\la a,x\ra-b)|,\]
where \an{$p_\d$} is given in~\eqref{eq:sfun}.
\an{The first derivative of the function $p_\d$ (see~\eqref{eq:pderiv})
is nonnegative and satisfies $p'_\d(s)\le 1$.
Hence, $\|\nabla h_\d(x;a,b)\|\le 1$ (see~\eqref{eq:gradh}).}

Now, let $\d_1\ge \d_2$. \an{If $\d_1=\d_2$, then
the relation}
$\max_{x\in\R^n}\|\nabla h_{\d_1}(x;a,b)- \nabla h_{\d_2}(x;a,b)\|
\le \frac{\d_1-\d_2}{2\d_1}$ holds trivially, so assume that
$\d_1 >\d_2$.
By relation~\eqref{eq:gradh}, we have
\begin{equation}\label{eq:33}
\|\nabla h_{\d_1}(x;a,b)- \nabla h_{\d_2}(x;a,b)\|
=|p'_{\d_1}(s)-p'_{\d_2}(s)|.
\end{equation}
\an{By using the expression for $p'_{\d}$, as given in~\eqref{eq:pderiv}, we have that}
the difference $p'_{\d_1}(s)-p'_{\d_2}(s)$ is given by
\begin{align*}
p'_{\d_1}(s)-p'_{\d_2}(s)=
\begin{cases}
 0 &\text{ if }  s<-\d_1 \ \text{ or } \  s\ge\d_1,\\
\frac{s + \d_1}{2\d_1} &\text{ if }  -\d_1\le s <-\d_2\\
\frac{s + \d_1}{2\d_1} - \frac{s + \d_2}{2\d_2} &\text{ if }  -\d_2\le s <\d_2,\\
\frac{s + \d_1}{2\d_1}-1&\text{ if }  \d_2 \le s <\d_1.
 \end{cases}
\end{align*}
Hence,
\begin{align*}
p'_{\d_1}(s)-p'_{\d_2}(s)=
\begin{cases}
 0 &\text{ if }  s<-\d_1\ \text{ or } \  s\ge\d_1,\\
\frac{s + \d_1}{2\d_1} &\text{ if } -\d_1\le s <-\d_2\\
\frac{(\d_2-\d_1)s}{2\d_1\d_2} &\text{ if }  -\d_2\le s <\d_2,\\
\frac{s - \d_1}{2\d_1}&\text{ if }  \d_2 \le s <\d_1.
 \end{cases}
\end{align*}
For the cases $-\d_1\le s <-\d_2$ and $\d_2 \le s \le\d_1$,
we have
\[\max_{s\in [\d_1,\d_2)\cup[\d_2,\d_1)}|p'_{\d_1}(s)-p'_{\d_2}(s)|
=\frac{\d_1-\d_2}{2\d_1} .\]
For the case $-\d_2\le s <\d_2$, we have
\[\max_{s\in [-\d_2,\d_2)}|p'_{\d_1}(s)-p'_{\d_2}(s)|
=\max_{s\in [-\d_2,\d_2)}\left|\frac{(\d_2-\d_1)s}{2\d_1\d_2} \right|
=\frac{\d_1-\d_2}{2\d_1}.\]
Therefore, we have
\[\max_{s\in\R}|p'_{\d_1}(s)-p'_{\d_2}(s)|\le \frac{\d_1-\d_2}{2\d_1},\]
which when combined with relation~\eqref{eq:33}
implies the desired relation.
\end{proof}

Our next lemma provides an upper bound on $\|x_{k+1}^* - x_k^*\|$,
which is critical for establishing the convergence of the method later on.

\begin{lem}\label{lem-improve}
Let $f$ be strongly convex with a constant $\mu>0$.
Let $\{\g_k\}$ and $\{\d_k\}$ be sequences of positive scalars, such that
\[\g_{k+1}\ge \g_k,\qquad \d_{k+1}\le \d_k \quad\hbox{ for all }k\ge1.\]
Then, we have for all $k\ge 1$
\begin{align*}
\mu\|x_k^*-x_{k+1}^*\|
\le (\g_{k+1}-\g_k)+\g_k\,\frac{\d_k-\d_{k+1}}{2\d_k}.
\end{align*}
\end{lem}
\begin{proof}
Consider an arbitrary $k\ge1$ and assume without loss of generality that
$x_k^*\ne x_{k+1}^*$ (for otherwise the stated relation holds trivially).
The optimality conditions
$\nabla F_k(x_k^*)=0$ and $\nabla F_{k+1}(x_{k+1}^*)=0$
yield, respectively,
\[ \nabla f(x_k^*) + \frac{\g_k}{m} \sum_{i=1}^{m} \nabla h_k\left(x_k^*; a_i, b_i\right)=0,\]
\[ \nabla f(x_{k+1}^*) + \frac{\g_{k+1}}{m} \sum_{i=1}^{m} \nabla h_{k+1}\left(x_{k+1}^*; a_i, b_i\right)=0.\]
By subtracting the last relation from the preceding one, and
by re-arranging the terms, we obtain
\begin{align*}
 \nabla f(x_k^*) - \nabla f(x_{k+1}^*)
= &\frac{\g_{k+1}}{m} \sum_{i=1}^{m} \nabla h_{k+1}\left(x_{k+1}^*; a_i, b_i\right)\cr&
-\frac{\g_k}{m} \sum_{i=1}^{m} \nabla h_k\left(x_k^*; a_i, b_i\right).
\end{align*}
 By adding and subtracting
$\frac{\g_k}{m} \sum_{i=1}^{m} \nabla h_{k+1}\left(x_{k+1}^*; a_i, b_i\right)$, we have
\begin{align*}
&\nabla f(x_k^*) - \nabla f(x_{k+1}^*)
= \frac{\g_{k+1}-\g_k}{m} \sum_{i=1}^{m} \nabla h_{k+1}\left(x_{k+1}^*; a_i, b_i\right)\cr
&+\frac{\g_k}{m} \sum_{i=1}^{m}
\left(\nabla h_{k+1}\left(x_{k+1}^*; a_i, b_i\right)
-\nabla h_k\left(x_k^*; a_i, b_i\right)\right).\end{align*}
Hence,
\begin{align*}
&\la \nabla f(x_k^*) - \nabla f(x_{k+1}^*),x_k^*-x_{k+1}^*\ra\cr
&= \frac{\g_{k+1}-\g_k}{m} \sum_{i=1}^{m}
\la \nabla h_{k+1}\left(x_{k+1}^*; a_i, b_i\right),x_k^*-x_{k+1}^*\ra +\frac{\g_k}{m}  \cr
&\times\sum_{i=1}^{m}
\la \nabla h_{k+1}\left(x_{k+1}^*; a_i, b_i\right)
-\nabla h_k\left(x_k^*; a_i, b_i\right), x_k^*-x_{k+1}^*\ra.\end{align*}
By the strong convexity of $f$, it follows that
\begin{align*}
&\mu\|x_k^*-x_{k+1}^*\|^2\cr
&\le \frac{\g_{k+1}-\g_k}{m} \sum_{i=1}^{m}
\la \nabla h_{k+1}\left(x_{k+1}^*; a_i, b_i\right),x_k^*-x_{k+1}^*\ra +\frac{\g_k}{m} \cr
&\times\sum_{i=1}^{m}
\la \nabla h_{k+1}\left(x_{k+1}^*; a_i, b_i\right)
-\nabla h_k\left(x_k^*; a_i, b_i\right), x_k^*-x_{k+1}^*\ra.\end{align*}
By adding and subtracting $\nabla h_{k+1}\left(x_k^*; a_i, b_i\right)$ in the last terms, we obtain
\begin{align*}
&\mu\|x_k^*-x_{k+1}^*\|^2\cr
&\le \frac{\g_{k+1}-\g_k}{m} \sum_{i=1}^{m}
\la \nabla h_{k+1}\left(x_{k+1}^*; a_i, b_i\right),x_k^*-x_{k+1}^*\ra +\frac{\g_k}{m}\cr
&\times \sum_{i=1}^{m}
\la \nabla h_{k+1}\left(x_{k+1}^*; a_i, b_i\right) - \nabla h_{k+1}\left(x_k^*; a_i, b_i\right), x_k^*-x_{k+1}^*\ra\cr
&+\frac{\g_k}{m} \sum_{i=1}^{m}
\la \nabla h_{k+1}\left(x_k^*; a_i, b_i\right)
-\nabla h_k\left(x_k^*; a_i, b_i\right), x_k^*-x_{k+1}^*\ra.\end{align*}
By the convexity of $h_{k+1}$, we have for all $i$,
\begin{align*}
\la \nabla h_{k+1}\left(x_{k+1}^*; a_i, b_i\right) - \nabla h_{k+1}\left(x_k^*; a_i, b_i\right),
x_k^*-x_{k+1}^*\ra\le0,
\end{align*}
implying that
\begin{align*}
&\mu\|x_k^*-x_{k+1}^*\|^2\cr
&\le \frac{\g_{k+1}-\g_k}{m} \sum_{i=1}^{m}
\la \nabla h_{k+1}\left(x_{k+1}^*; a_i, b_i\right),x_k^*-x_{k+1}^*\ra \cr
&+\frac{\g_k}{m} \sum_{i=1}^{m}
\la \nabla h_{k+1}\left(x_k^*; a_i, b_i\right)
-\nabla h_k\left(x_k^*; a_i, b_i\right), x_k^*-x_{k+1}^*\ra.\end{align*}
Since $\g_{k+1}\ge \g_k>0$, by using Cauchy-Schwarz inequality, we obtain
\begin{align*}
&\mu\|x_k^*-x_{k+1}^*\|^2\cr
&\le \frac{\g_{k+1}-\g_k}{m} \sum_{i=1}^{m}
\|\nabla h_{k+1}\left(x_{k+1}^*; a_i, b_i\right)\|\,\|x_k^*-x_{k+1}^*\|\cr
&+\frac{\g_k}{m} \sum_{i=1}^{m}
\|\nabla h_{k+1}\left(x_k^*; a_i, b_i\right)-\nabla h_k\left(x_k^*; a_i, b_i\right)\|\cr
&\qquad\qquad\qquad\qquad\qquad\times \|x_k^*-x_{k+1}^*\|.
\end{align*}
By Lemma~\ref{lem:pderiv}, we have that
$\|\nabla h_{k+1}\left(x_{k+1}^*; a_i, b_i\right)\|\le 1$ implying that
\begin{align*}
&\mu\|x_k^*-x_{k+1}^*\|^2
\le (\g_{k+1}-\g_k)\|x_k^*-x_{k+1}^*\|+\frac{\g_k}{m}\cr
& \times\sum_{i=1}^{m}
\|\nabla h_{k+1}\left(x_k^*; a_i, b_i\right)-\nabla h_k\left(x_k^*; a_i, b_i\right)\|\, \|x_k^*-x_{k+1}^*\|.
\end{align*}
Since $\d_{k+1}\le \d_k$, by Lemma~\ref{lem:pderiv} we have that
for all $i$,
\[\max_{x\in\R^n}\|\nabla h_{\d_k}(x;a_i,b_i)- \nabla h_{\d_{k+1}}(x;a_i,b_i)\|
\le \frac{\d_k-\d_{k+1}}{2\d_k}.\]
Hence,
\begin{align*}
&\mu\|x_k^*-x_{k+1}^*\|^2\cr
&\le (\g_{k+1}-\g_k)\|x_k^*-x_{k+1}^*\|+\g_k\,\frac{\d_k-\d_{k+1}}{2\d_k} \|x_k^*-x_{k+1}^*\|.
\end{align*}
Dividing by $\|x_k^*-x_{k+1}^*\|$, we obtain
\begin{align*}
\mu\|x_k^*-x_{k+1}^*\|
\le (\g_{k+1}-\g_k)+\g_k\,\frac{\d_k-\d_{k+1}}{2\d_k}.
\end{align*}
\end{proof}

Our next result provides relations for the points $x_k^*$ and
the optimal solution $x^*$ of the original problem.

\begin{lem}\label{lema-solsbded}
Let $f$ be strongly convex with a constant \hbox{$\mu>0$}.
Assume that the sequence $\{\d_k\}$ and $\{\g_k\}$ are such that $\g_k>0$, $\d_k>0$
and
$\g_k\d_k\le c$ for all $k$.
Then, the sequence $\{x_k^*\}$ of solutions (to the
corresponding penalized problems $\min_{x\in \R^n} F_k(x)$)
is contained in the level set
\[\left\{x\in\R^n\mid f(x)\le f(x^*)+\frac{c}{4\a_{\min}}\right\},\]
where $x^*$ is the solution of the original problem and
$\a_{\min}=\min_{1\le i\le m}\|a_i\|$.
In particular, the sequence $\{x_k^*\}$ is bounded.
\end{lem}
\begin{proof}
By Corollary~\ref{cor:lset}, where $\hat x=x^*$,
\hbox{$\d=\d_k$, $\g=\g_k$,}
and
$t_{\g\d}(\hat x)$ is replaced with
$f(x^*)+\frac{\g_k\d_k}{4\a_{\min}}$, we obtain
that
$f(x_k^*)\le f(x^*)+\frac{\g_k\d_k}{4\a_{\min}}$. Since \hbox{$\g_k\d_k\le c$,}
it follows that $\{x^*_k\}$ is contained in
the level set $\left\{x\in\R^n\mid f(x)\le f(x^*)+\frac{c}{4\a_{\min}}\right\}$.
Since $f$ is strongly convex, its level sets are bounded, thus implying that
$\{x_k^*\}$ is bounded.
\end{proof}

We next consider a set of conditions on parameters $\d_k$ and $\g_k$ that will ensure
that the sequence $\{x_k^*\}$ converges to $x^*$ as $k\to\infty$.
In what follows, we will use the projections of
the points $x^*_k$ on the feasible set, which
we denote by $p_k$, i.e., $p_k=\Pi_X[x_k^*]$.
Under the assumptions of Lemma~\ref{lema-solsbded},
the sequence$\{x_k^*\}$ is bounded, and so is the sequence
$\{p_k\}$ of the projections of $x_k^*$'s on $X$.
Let $R$ be large enough so that $\{x_k\}\subset \mathbb{B}(0,R)$ and
$\{p_k\}\subset \mathbb{B}(0,R)$, where $\mathbb{B}(0,R)$ denotes the ball centered at the
origin with the radius $R$.
The subgradients of $f(x)$ for $x\in \mathbb{B}(0,R)$ are bounded, and let
$L$ be the maximum norm
of the subgradients of $f(x)$ over $x\in \mathbb{B}(0,R)$,
i.e.,
\begin{equation}\label{eq:sgdbound}
L=\max\{\|s\|\mid s\in\partial f(x), \|x\|\le R\}.
\end{equation}

We have the following lemma.
\begin{lem}\label{lema-sols}
Let $f$ be strongly convex with a constant $\mu>0$.
Assume that the sequence $\{\d_k\}$ and $\{\g_k\}$ are such that $\g_k>0$, $\d_k>0$
and $\g_k\d_k\le c$ for all $k$. Let $L$ be given by~\eqref{eq:sgdbound}. Then,
for all $k$, we have
\begin{align*}
 \frac{\mu}{2}\|x^* -x_k^*\|^2 + \frac{\mu}{2}\|x^*-p_k\|^2
+\left(\frac{\g_k}{4m\b} - L\right) &\dist(x_k^*,X)\cr
&\le \frac{\g_k\d_k}{4\a_{\min}},
\end{align*}
where $p_k=\Pi_X[x_k^*]$ for all $k$, and
$\a_{\min}=\min_{1\le i\le m}\|a_i\|$.
\end{lem}
\begin{proof}
Since each $F_k$ is strongly convex with the constant $\mu>0$,
by the optimality of $x_k^*$ and by the definition of $F_k$, we have
\begin{align}\label{eq:11}
\frac{\mu}{2}&\|x^* -x_k^*\|^2\le 
f(x^*)-f(x_k^*) \cr
&+ \frac{\g_k}{m} \sum_{i=1}^{m} h_k\left(x^*; a_i, b_i\right)
-\frac{\g_k}{m} \sum_{i=1}^{m} h_k\left(x_k^*; a_i, b_i\right).
\end{align}
By adding and subtracting $f(p_k)$,
we have
\begin{align*}
 f(x^*)-f(x_k^*)=f(x^*)-f(p_k) + f(p_k) - f(x_k^*)\cr\le -\frac{\mu}{2}\|x^*-p_k\|^2
+|f(p_k)-f(x_k^*)|,
\end{align*}
where in the last inequality is obtained using
\[\frac{\mu}{2}\|p_k-x^*\|^2 + \la \nabla f(x^*), p_k-x^*\ra +f(x^*)\le f(p_k),\]
and the fact $\la \nabla f(x^*), p_k-x^*\ra\ge0$, which holds since $p_k$ is feasible
and $x^*$ is the optimal point.
Since the subgradients of $f$ are uniformly bounded at $x^*_k$ and $p_k$,
we have
\[f(x^*)-f(x_k^*)\le -\frac{\mu}{2}\|x^*-p_k\|^2 + L \|p_k - x_k^*\|.\]
Combining the preceding inequality with relation~\eqref{eq:11}, we obtain
\begin{align*}
\frac{\mu}{2}&\|x^* -x_k^*\|^2\le -\frac{\mu}{2}\|x^*-p_k\|^2 + L \|p_k - x_k^*\|\cr
&+ \frac{\g_k}{m} \sum_{i=1}^{m} h_k\left(x^*; a_i, b_i\right)
-\frac{\g_k}{m} \sum_{i=1}^{m} h_k\left(x_k^*; a_i, b_i\right).
\end{align*}
Since $x^*$ is feasible, by relation~\eqref{eq:hfunineq1} (where $\d=\d_k$)
we have for all $i=1,\ldots,m$,
\[h_k(x^*; a_i,b_i) \le \frac{\d_k}{4\|a_i\|}\le \frac{\d_k}{4\a_{\min}}.\]
Hence, it follows that
\begin{align*}
 \frac{\mu}{2}\|x^* -x_k^*\|^2 + \frac{\mu}{2}\|x^*-p_k\|^2 \le& L \|p_k - x_k^*\|
+ \frac{\g_k\d_k}{4\a_{\min}}\cr
&-\frac{\g_k}{m} \sum_{i=1}^{m} h_k\left(x_k^*; a_i, b_i\right).
\end{align*}

By Lemma~\ref{lem:penalty}, where $\d=\d_k$, $a=a_i$, $b=b_i$, and $Y=X_i$,
we have
\[h_k(x_k^*;a_i,b_i)\ge \frac{1}{4}\dist(x_k^*,X_i)
\qquad\hbox{for all $i=1,\ldots,m$}.\]
Therefore,
\begin{align*}
 \frac{\mu}{2}\|x^* -x_k^*\|^2 + \frac{\mu}{2}\|x^*-p_k\|^2 \le& L \|p_k - x_k^*\|
+  \frac{\g_k\d_k}{4\a_{\min}}\cr
&-\frac{\g_k}{4m} \sum_{i=1}^{m} \dist(x_k^*,X_i).
\end{align*}
By Hoffman's lemma~\cite{Hoffman},
we have that there exists
$\beta=\beta(a_1,\ldots,a_m)>0$ for the sets $X_i$, $i=1,\ldots,m$, such that
$$
\beta \sum_{i=1}^{m} \dist(x, X_i) \ge \dist(x,X)\quad \hbox{for all } x\in\R^n.
$$
Hence,
\begin{align*}
\frac{\mu}{2}\|x^* -x_k^*\|^2 + \frac{\mu}{2}\|x^*-p_k\|^2 \le& L \|p_k - x_k^*\|
+ \frac{\g_k\d_k}{4\a_{\min}}\cr
&-\frac{\g_k}{4m\b} \dist(x_k^*,X).
\end{align*}
Since $p_k$ is the projection of $x_k^*$ on $X$, we have
$\|p_k-x_k^*\|=\dist(x_k^*,X)$ implying that
\begin{align*}\frac{\mu}{2}\|x^* -x_k^*\|^2 &+ \frac{\mu}{2}\|x^*-p_k\|^2
\cr&\le \left(L -\frac{\g_k}{4m\b}\right) \dist(x_k^*,X)
+ \frac{\g_k\d_k}{4\a_{\min}}.
\end{align*}
\end{proof}

Lemma~\ref{lema-sols} indicates that, when $\g_k\to+\infty$,
for all large enough $k$, we will have $\frac{\g_k}{4m\b}>0$,
implying that
\begin{align*}
\dist(x_k^*,X)&\le \frac{\g_k\d_k}{4\a_{\min} \left(\frac{\g_k}{4m\b}-L\right)}\cr
&=\frac{m\b\g_k\d_k}{\a_{\min}(\g_k-4m\b L)}
\approx O(\d_k).
\end{align*}
Thus, if $\d_k\to0$,  the distance of $x_k^*$ to the feasible set $X$ will go to 0
at the rate of  $O(\d_k)$.
Lemma~\ref{lema-sols} also indicates that, for large enough $k,$
\[\|x^* -x_k^*\|^2\le  \frac{\g_k\d_k}{2\mu\a_{\min}}.\]
Thus, if $\g_k\d_k\to0$, then the points $x_k^*$ approach
the optimal solution $x^*$ of the original problem, with the rate of
$O(\g_k\d_k)$.

To summarize, Lemma~\ref{lema-sols} characterizes the behavior of the sequence $\{x_k^*\}$
in terms of the penalty parameters $\{\g_k\}$ and $\{\d_k\}$. It shows that
under conditions $\g_k\to\infty$, $\d_k\to0$ and $\g_k\d_k\to0$, we have
$\|x_k^*-x^*\|\to0$. Based on Lemma~\ref{lema-sols}, one can construct a two-loop approach to compute
the optimal point $x^*$ of the original problem,
where for every outer loop $k$, we have an inner loop of iterations to compute $x_k^*$. This, however,
will be quite inefficient. In the next section, we propose a more efficient single-loop algorithm,
where at each iteration $k$ we use the gradient of the penalty function $F_k$.

\section{Convergence rate of incremental gradient algorithm}
The results of Lemma~\ref{lem-improve} and Lemma~\ref{lema-sols} are useful
for analyzing the convergence behavior of an incremental algorithm that,
when the iterate $x_k$ is available at iteration $k$,
uses only one randomly chosen constraint (indexed by $i_k$) to estimate
the gradient $\nabla F_k(x_k)$. \an{This estimation is employed to construct $x_{k+1}$},
as opposed to determining $x^*_k$ for each function $F_k$.
We illustrate this on a simple incremental gradient-based method, given by:
for $k\ge 1$,
\begin{equation}\label{eq:gradmet}
x_{k+1} = x_k-s_k [\nabla f(x_k) + \g_k\nabla h_k(x;a_{i_k},b_{i_k})],
\end{equation}
where $x_1$ is an initial point, $s_k>0$ is a stepsize, and the index $i_k$ is chosen uniformly at random.
Note that $\nabla f(x_k) + \g_k\nabla h_k(x;a_{i_k},b_{i_k})$ is an unbiased estimation of $\nabla F_k(x_k)$, \an{since by the choice of $i_k$ we have}
\[\ex{\nabla f(x_k) + \g_k\nabla h_k(x;a_{i_k},b_{i_k})}=\nabla F_k(x_k)
\ \ \hbox{for  $k\ge1$.}\]

The idea behind the analysis of the method~\eqref{eq:gradmet} is resting on a relation of the form
$\ex{\|x_{k+1} - x^*\|}\le q_k \ex{\|x_k - x^*\|}+ r_k$
for some $q_k$ and $r_k$ and explores the conditions on $q_k$ and
$r_k$, for which the following Chung's lemma \cite{Chung} ensures the convergence
of $\|x_k-x^*\|$ to 0, as $k\to\infty$ with some definite convergence rate.

\begin{lem}\label{lema-chung}
Let $\{u_k\}$ be a nonnegative scalar sequence \an{and $k_0$ be such that}
\[u_{k+1}\le \left(1-\frac{a}{k^s}\right) u_k+O\left(\frac{b}{k^{s+t}}\right)\qquad\hbox{ for all $k>k_0$,}\]
where $0<s\le 1$, $a>0$, $b>0$, and $t>0$.
Then, we have
\begin{align*}
u_k= O\left(\frac{1}{k^{t}}\right).
\end{align*}
\end{lem}

With Lemma~\ref{lem-improve}, Lemma~\ref{lema-sols}, and Lemma~\ref{lema-chung} in place,
we next establish a set of conditions on $\{\g_k\}$ and $\{\d_k\}$
that ensure convergence of the iterates produced by the method~\eqref{eq:gradmet}.

\begin{prop}\label{prop:gradconv}
Let $f$ be strongly convex with a constant \hbox{$\mu>0$} and have Lipschitz continuous gradients
with a constant $L_f$.
Let the sequences $\{\g_k\}$ and $\{\d_k\}$ satisfy
\[\g_k = k^g, \, \d_k = \frac{1}{k^d},\]
where \an{$g>0$ and $d>0$ are such that $\{\g_k\d_k\}$ is nonincreasing}.
Consider the method~\eqref{eq:gradmet} with the stepsize $s_k = \frac{1}{k^s}$ with $s>0$.
Then, as $k\to\infty$,
\[\an{\E\|x_k - x^*\|^2} = O\left(\frac{1}{k^{\min\{s-2g,2-2s+2g\}}} + \frac{1}{k^{d-g}}\right).\]
In particular, when $s=1$, $g=\frac{1}{4}$, and $d\ge \frac{3}{4}$ the iterates $\{x_k\}$ the method~\eqref{eq:gradmet} converge to the solution $x^*$
of the original problem \an{(in expectation)} and
\[\an{\E\|x_k - x^*\|^2} = O\left(\frac{1}{k^{\frac{1}{2}}}\right).\]
\end{prop}
\begin{proof}
For any $k\ge0$, for the iterates of the method we have
\begin{align*}
 \|x_{k+1}- x_k^*\|^2 =\|x_k - x_k^*\|^2 &-2s_k\la g_k(x_k),x_k-x_k^*\ra\cr
 &+s_k^2\| g_k(x_k)\|^2,
\end{align*}
where $g_k(x_k)=\nabla f(x_k) + \g_k\nabla h_k(x;a_{i_k},b_{i_k})$.
By the strong convexity of $F_k$ and the fact $\nabla F_k(x_k^*)=0$, it follows that
\begin{equation}
\label{eq-main1}
\E\|x_{k+1}- x_k^*\|^2 \le (1-2s_k\mu)\E\|x_k - x_k^*\|^2 +s_k^2\E\|g_k(x_k)\|^2.
\end{equation}
For $\|g_k(x_k)\|^2$
we write
\begin{align*}\E\|g_k(x_k)\|^2
\le 2\E\|\nabla f(x_k)\|^2 +2\E\left\|\g_k \nabla h_k\left(x_k; a_{i_k}, b_{i_k}\right) \right\|\cr
 \le 2\E\|\nabla f(x_k)\|^2 +2\g_k^2,
\end{align*}
where the last inequality is obtained by using the convexity of the squared-norm function
and the fact that $\|\nabla h_k\left(x; a_i, b_i\right) \|\le 1$ for any $x$ and $i$ (see Lemma~\ref{lem:pderiv}).
We further estimate $\E\|\nabla f(x_k)\|^2$ as follows:
\begin{align*}
\E\|\nabla f(x_k)\|^2\le &2\E\|\nabla f(x_k)-\nabla f(x^*)\|^2 +2\|\nabla f(x^*)\|^2\cr
\le &2L_f^2\E\|x_k-x^*\|^2 +2\|\nabla f(x^*)\|^2,
\end{align*}
where in the last inequality we use the Lipschitz gradient property of $f$.
Thus,
\[\E\|g_k(x_k)\|^2\le 4L_f^2\E\|x_k-x^*\|^2 +4\|\nabla f(x^*)\|^2 +2\g_k^2.\]
Further, we have
\[\E\|x_k-x^*\|^2\le 2\E\|x_k-x_k^*\|^2 +2\|x_k^*-x^*\|^2,\] so that
\begin{align*}
 \E\|g_k(x_k)\|^2\le 8L_f^2\E\|x_k-x_k^*\|^2&+ 8L_f^2\|x_k^*-x^*\|^2 \cr
 &+4\|\nabla f(x^*)\|^2 +2\g_k^2.
\end{align*}
By Lemma~\ref{lema-sols} we have
\begin{equation}\label{eq-estimk}
\|x_k^*-x^*\|^2\le \frac{\g_k\d_k}{2\mu\a_{\min}}.\end{equation}
By combining the preceding two relations with relation~\eqref{eq-main1}
we obtain
\begin{align*}
\E\|x_{k+1}- x_k^*\|^2
\le &(1-2s_k\mu +8L_f^2 s_k^2)\E\|x_k - x_k^*\|^2 \cr
&+s_k^2\left(\frac{4L_f^2   \g_k\d_k}{\mu\a_{\min}}+4\|\nabla f(x^*)\|^2 +2\g_k^2\right).
\end{align*}
We next consider $\|x_{k+1}- x_{k+1}^*\|^2$ for which we write
\begin{align*}\|x_{k+1}- x_{k+1}^*\|^2 \le &(1+s_k\mu)\|x_{k+1}- x_k^*\|^2 \cr
&+(1+s_k^{-1}\mu^{-1})\|x_k^* - x_{k+1}^*\|^2.\end{align*}
Combining the preceding two relations, we obtain
\begin{align}\label{eq-main2}
&\E\|x_{k+1}- x_{k+1}^*\|^2\cr
&\le (1+s_k\mu) (1-2s_k\mu +8L_f^2 s_k^2)\E\|x_k - x_k^*\|^2 \cr
&+(1+s_k\mu)s_k^2\left(\frac{2L_f^2   \g_k\d_k}{\mu\a_{\min}}+4\|\nabla f(x^*)\|^2 +2\g_k^2\right)\cr
&+ (1+s_k^{-1}\mu^{-1})\|x_k^* - x_{k+1}^*\|^2.
\end{align}

Next we use Lemma~\ref{lem-improve} to upper bound $\|x_k^*- x_{k+1}^*\|^2$. Thus, we obtain for large enough $k$,
\begin{align}\label{eq-main3}
&\E\|x_{k+1}- x_{k+1}^*\|^2\cr
&\le (1+s_k\mu) (1-2s_k\mu +8L_f^2 s_k^2)\E\|x_k - x_k^*\|^2 \cr
&\quad+(1+s_k\mu)s_k^2\left(\frac{2L_f^2   \g_k\d_k}{\mu\a_{\min}}+4\|\nabla f(x^*)\|^2 +2\g_k^2\right)\cr
&\quad+ \frac{1+s_k^{-1}\mu^{-1}}{\mu^2}\left(\g_{k+1}-\g_k+\g_k\,\frac{\d_k-\d_{k+1}}{2\d_k}\right)^2.
\end{align}

The rest of the proof is verifying that Lemma~\ref{lema-chung} can be applied to the
preceding inequality. Indeed, let
\[u_k=\E\|x_k - x_k^*\|^2,\qquad
q_k=(1+s_k\mu) (1-2s_k\mu +8L_f^2 s_k^2),\]
\begin{align*}
r_k=&(1+s_k\mu)s_k^2\left(\frac{2L_f^2   \g_k\d_k}{\mu\a_{\min}}+4\|\nabla f(x^*)\|^2 +2\g_k^2\right)\cr
&+  \frac{1+s_k^{-1}\mu^{-1}}{\mu^2}\left(\g_{k+1}-\g_k+\g_k\,\frac{\d_k-\d_{k+1}}{2\d_k}\right)^2.
\end{align*}
Consider the coefficient $q_k$, for which we have for sufficiently large $k$,
\begin{align*}
q_k
&=1-2s_k\mu +8L_f ^2s_k^2
+s_k\mu -2s_k^2\mu^2+ 8L_f^2 \mu s_k^3\cr
&= 1- s_k\mu +(8L_f^2 -2\mu^2)s_k^2
+ 8L_f^2 \mu s_k^3\cr
&\ge 1- \frac{\mu}{2}s_k,\end{align*}
where in the last inequality we use the fact that $s_k \to 0$ as $k\to\infty$.
For the coefficient $r_k$, since $\g_k\d_k$ is nonincreasing and $\g_k\to\infty$
we have for large enough $k$,
\[r_k\approx O(s_k^2\g_k^2) +O\left(\frac{\left(\g_{k+1}-\g_k+\g_k\,\frac{\d_k-\d_{k+1}}{2\d_k}\right)^2}{s_k}\right).\]
Next, taking into account the settings $s_k = \frac{1}{k^s}$, $\g = {k^g}$, $\d = \frac{1}{k^d}$, we obtain
\begin{align*}
  &u_{k+1}\le \left(1-\frac{\mu}{2}\frac{1}{k^s}\right) u_k + O\left(\frac{1}{k^{2s-2g}}\right) \\
   &  +O\left(\frac{1}{k^{2g-s}}\left[\left(1+\frac{1}{k}\right)^g-1+\frac{1-\left(1-\frac{1}{k+1}\right)^d}{2}\right]^2\right).
\end{align*}
Due to the fact that $\left(1+\frac{1}{k}\right)^g\equiv 1 + \frac{g}{k}$ and $\left(1-\frac{1}{k+1}\right)^d\equiv 1 - \frac{d}{k+1}$, we conclude that
\begin{align*}
  u_{k+1}\le \left(1-\frac{\mu}{2}\frac{1}{k^s}\right) u_k + O\left(\frac{1}{k^{2s-2g}} + \frac{1}{k^{2 -s+2g}}\right). \\
\end{align*}
Next, we write
\[\|x_k-x^*\|^2\le 2\|x_k-x_k^*\|^2 + 2\|x_k^*-x^*\|^2,\]
which together with \eqref{eq-estimk} implies
\[\E\|x_k - x^*\|^2 = O\left(\frac{1}{k^{\min\{s-2g,2-2s+2g\}}} + \frac{1}{k^{d-g}}\right).\]
By optimizing the parameters $s$, $g$, and $d$ , we get $s=1$, $g=\frac{1}{4}$, and $d\ge \frac{3}{4}$. Under this setting
\[\E\|x_k - x^*\|^2 = O\left(\frac{1}{k^{\frac{1}{2}}}\right),\]
and the iterates $\{x_k\}$ the method~\eqref{eq:gradmet}
converge, \an{in the expectation}, to the solution $x^*$ of the original problem.

\end{proof}




\section{Conclusion}
In this work we considered penalty reformulation of optimization problems
with strongly convex objectives and linear constraints.
We proposed using Huber losses as penalty functions.
The properties of these functions allowed us to set up the penalty parameter and the step-size of the standard incremental gradient-based optimization procedure to guarantee convergence to the solution.
Moreover, we provided the estimation of the convergence rate for this algorithm.
In the future work, we will investigate applicability of accelerated incremental algorithms for
the proposed penalty reformulation in the case of both strongly and non-strongly convex optimization.
 \bibliographystyle{plain}
\bibliography{Literature}

\end{document}